\documentclass[12pt]{amsart}
\usepackage[shortlabels]{enumitem}
\usepackage{amsmath,amscd,amsthm,amsfonts,amssymb}
\usepackage{eucal}
\usepackage{mathrsfs}
\usepackage{float,morefloats}
\usepackage[margin=1in]{geometry}
\usepackage{tikz-cd, comment, fancyvrb}
\usepackage{todonotes}
\usepackage{xcolor}
\usepackage{bbold}

\RequirePackage{color}
\definecolor{myred}{rgb}{0.75,0,0}
\definecolor{mygreen}{rgb}{0,0.5,0}
\definecolor{myblue}{rgb}{0,0,0.65}

\usepackage{hyperref}
\usepackage[nameinlink]{cleveref}
\hypersetup{citecolor=blue}
\usepackage{tikz}
\usetikzlibrary{matrix,arrows,decorations.pathmorphing}

\let\temp\emptyset
\let\emptyset\varnothing
\let\varnothing\temp 

\theoremstyle{plain}
\newtheorem{theorem}{Theorem}
\newtheorem{proposition}[theorem]{Proposition}

\newtheorem*{lemma*}{Lemma}
\newtheorem*{proposition*}{Proposition}

\theoremstyle{definition}

\theoremstyle{remark}

\newcommand{\bb}[1]{\expandafter\newcommand\expandafter{\csname #1\endcsname}{{\mathbb {#1}}}} 

\bb C
\bb R
\bb Q
\bb Z
\bb N

\bb F

\newcommand{\ZZ}{\mathbb Z}
\newcommand{\NN}{\mathbb N}

\renewcommand{\a}{\alpha}

\newcommand{\mrm}[1]{\expandafter\newcommand\expandafter{\csname #1\endcsname}{{\mathrm {#1}}}}

\mrm{Hom}
\mrm{Aut}
\mrm{End}
\mrm{Spec}
\mrm{Gal}

\mrm{PGL}
\mrm{PSL}
\mrm{GL}
\mrm{SL}
\mrm{PG}
\mrm{PSO}
\mrm{PS}
\mrm{PSU}

\mrm{lcm}

\renewcommand{\mod}{\bmod}

\global\long\def\NN{\mathbb{N}}
\global\long\def\ZZ{\mathbb{Z}}

\newtheorem{case}{Case}

\title{Consecutive runs of sums of two squares}
\author{Noam Kimmel}
\author{Vivian Kuperberg}
\thanks{ This research was supported by the European Research Council (ERC) under the European Union's  Horizon 2020 research and innovation program  (Grant agreement No. 786758). The second author was supported by the NSF Mathematical Sciences Research Program through the grant DMS-2202128.}
\keywords{Sums of two squares, Sieve methods, Arithmetic progressions}

\begin{document}

\begin{abstract}
    We study the distribution of consecutive sums of two squares in arithmetic progressions. If $\{E_n\}_{n \in \N}$ is the sequence of sums of two squares in increasing order, we show that for any modulus $q$ and any congruence classes $a_1,a_2,a_3 \mod q$ which are admissible in the sense that there are solutions to $x^2 + y^2 \equiv a_i \mod q$, there exist infinitely many $n$ with $E_{n+i-1} \equiv a_i \mod q$, for $i = 1,2,3$. We also show that for any $r_1, r_2 \ge 1$, there exist infinitely many $n$ with $E_{n+i-1} \equiv a_1 \mod q$ for $1 \le i \le r_1$ and $E_{n+ i - 1} \equiv a_2 \mod q$ for $r_1 + 1 \le i \le r_1 + r_2$.
\end{abstract}
\maketitle

\section{Introduction}

Dirichlet's theorem on primes in arithmetic progressions states that for any fixed modulus $q$ and any congruence class $a$ mod $q$ with $(a,q) = 1$, there are infinitely many primes congruent to $a$ mod $q$. The prime number theorem for arithmetic progressions goes further, stating that the sequence of primes is equidistributed among the reduced residue classes of a modulus $q$; i.e., if $\pi(x;q,a)$ denotes the number of primes $p \le x$ which are congruent to $a$ mod $q$ and $(a,q) = 1$, then
\[\pi(x;q,a) = \frac{\mathrm{li}(x)}{\phi(q)}(1 + o(1)),\]
where $\mathrm{li}(x) := \int_2^x \frac{\mathrm dy}{\ln y}$ is the logarithmic integral. 

One can ask the same question about strings of consecutive primes, where much less is known. Let $p_n$ denote the sequence of primes in ascending order. For a fixed modulus $q$ and an integer $r \ge 1$, for any $r$-tuple $\mathbf a = [a_1, \dots, a_r]$ of reduced residue classes mod $q$, let $\pi(x;q,\mathbf a)$ denote the number of strings of consecutive primes matching the residue classes of $\mathbf a$. Precisely, define
\begin{equation}
    \pi(x;q,\mathbf a) := \#\{p_n \le x: p_{n+i-1} \equiv a_i \pmod q \qquad \forall 1 \le i \le r\}.
\end{equation}
Any randomness-based model of the primes would suggest that sequences of consecutive primes equidistribute among the possibilities for $\mathbf a$, i.e. that $\pi(x;q,\mathbf a) \sim \frac{\pi(x)}{\phi(q)^r}$ as $x \to \infty$. Lemke Oliver and Soundararajan \cite{MR3624386-lemke-oliver-sound} provide a heuristic argument based on the Hardy--Littlewood $k$-tuples conjectures for estimating $\pi(x;q,\mathbf a)$, conjecturing that the asymptotic density is in fact $\frac 1{\phi(q)^r}$, but that there are large second-order terms creating biases among the patterns.

However, little is known about $\pi(x;q,\mathbf a)$ whenever $r \ge 2$; in most cases, it is not even known that $\pi(x;q,\mathbf a)$ tends to infinity with $x$---i.e., that there are infinitely many strings of consecutive primes $p_n, \dots, p_{n+r-1}$ with $p_{n+i-1} \equiv a_i \mod q$. If $\phi(q) = 2$, it is an immediate consequence of Dirichlet's theorem on primes in arithmetic progressions that for reduced residues $a \ne b \mod q$, $\pi(x;q,[a,b])$ and $\pi(x;q,[b,a])$ tend to infinity; Knapowski and Tur\'an \cite{MR0466043-knapowski-turan} observed that when $\phi(q) = 2$, all four possible patterns of length $2$ occur infinitely often. Shiu \cite{MR1760689-Shiu} showed more generally that when every element of $\mathbf a$ is the same residue class mod $q$, so that $\mathbf a = [a,\dots,a]$ for some reduced $a$ mod $q$, then $\pi(x;q,[a,\dots,a])$ tends to infinity for all $r$ and for all reduced $a$ mod $q$. This result was reproven by Banks, Freiberg, and Turnage-Butterbaugh \cite{MR3316460-banks-freiberg-turnage-butterbaugh} using developments in sieve theory. Maynard \cite{MR3530450-Maynard-dense-clusters} used sieve methods to show further that $\pi(x;q,[a \dots, a]) \gg \pi(x)$, so that a positive density of primes begins strings of $r$ consecutive primes all of which are congruent to $a$ mod $q$. 

For all other cases---that is, for any $r \ge 2$ with $\phi(q) > 2$ or $r \ge 3$ with $\phi(q) = 2$ and any $\mathbf a$ which is non-constant---it is not known that $\pi(x;q,\mathbf a)$ tends to infinity with $x$.

In this paper, we consider the analogous question for sums of two squares. Let $\mathbf E$ denote the set of sums of two squares and let $\mathbb {1}_{\mathbf E}$ denote its indicator function. Let $E_n$ denote the increasing sequence of sums of two squares, so that 
\[\mathbf E = \{a^2 + b^2:a,b \in \mathbb Z\} = \{E_n:n \in \N\}. \]
Let $N(x)$ denote the number of sums of two squares less than $x$. A number $n$ is expressible as a sum of two squares if and only if every prime congruent to $3$ mod $4$ divides $n$ to an even power; that is, if $n$ factors as $n = \prod_p p^{e_p}$ then $e_p$ is even whenever $p \equiv 3 \mod 4$. For a modulus $q = \prod_p p^{e_p}$ and a congruence class $a$ mod $q$, write $(a,q) = \prod_p p^{f_p}$, where $f_p \le e_p$ for all $p$. There are infinitely many sums of two squares congruent to $a$ mod $q$ if and only if the following two conditions hold:
\begin{itemize}
    \item for any prime $p \equiv 3 \mod 4$, $f_p$ is either even or $f_p = e_p$, and 
    \item if $e_2 - f_2 \ge 2$, then $\frac{a}{2^{f_2}} \not\equiv 3 \mod 4$.
\end{itemize}
We will call a congruence class $a$ mod $q$ \emph{admissible} if it satisfies these conditions, i.e. if there exists a solution to $x^2 + y^2 \equiv a \bmod q$.

For a modulus $q$, an integer $r \ge 1$ and an $r$-tuple $\mathbf a = [a_1, \dots, a_r]$ of admissible residue classes mod $q$, let 
\begin{equation}
    N(x;q,\mathbf a) := \#\{E_n \le x: E_{n+i-1} \equiv a_i \mod q \quad \forall 1 \le i \le r\}.
\end{equation}
Just as in the prime case, one expects $N(x;q,\mathbf a)$ to tend to infinity for any tuple of admissible residue classes, and in fact one expects $N(x;q,\mathbf a) \gg N(x)$ where $N(x)$ is the number of sums of two squares less than $x$; that is, one expects $N(x;q,\mathbf a)$ to represent a positive proportion of sums of two squares. In the case when $q \equiv 1 \mod 4$ is a prime, David, Devin, Nam, and Schlitt \cite{MR4498471-david-devin-nam-schlitt} develop heuristics for second-order terms in the asymptotics of $N(x;q,\mathbf a)$ analogously to \cite{MR3624386-lemke-oliver-sound}. Their heuristics are based on the analog of the Hardy--Littlewood $k$-tuples conjecture in the setting of sums of two squares, which was developed in \cite{MR3733767-freiberg-kurlberg-rosenzweig}.

Unlike in the case for primes, the infinitude of $N(x;q,\mathbf a)$ has not been proven when $r \ge 2$. We prove for several specific cases that $N(x;q,\mathbf a)$ tends to infinity with $x$.

\begin{theorem}\label{thm:infinitude-aaabbb}
Let $q \ge 1$ be an odd modulus and let $a$ and $b$ be reduced residue classes modulo $q$. Let $r \ge 1$, and let $\mathbf a = [a_1, \dots, a_r]$ be a tuple of residue classes such that for some $1 \le j \le r$, $a_i = a$ whenever $i \le j$ and $a_i = b$ whenever $i > j$. Then $N(x;q,\mathbf a)$ tends to infinity with $x$.
\end{theorem}

\Cref{thm:infinitude-aaabbb} relies on stronger sieving results than those that are available in the prime case. In particular, McGrath \cite{MR4498475-McGrath} showed that for any $m$, for large enough $k$, any \emph{admissible} $k$-tuple $\{h_1, \dots, h_k\}$ can be divided into $m$ sub-tuples or ``bins,'' such that for infinitely many $n$, for some $h_i$ in each bin, $n+h_i$ is a sum of two squares. Here and throughout this paper, a $k$-tuple $\mathcal H = \{h_1, \dots, h_k\}$ is \emph{admissible} if for all primes $p$, $\#\{\mathcal H \pmod p\} < p$. Banks, Freiberg, and Maynard \cite{MR3556490-Banks-Freiberg-Maynard} used a similar result in the case of primes to show that a positive proportion of real numbers are limit points of the sequence of normalized prime gaps, work which was refined in \cite{MR3855375-Pintz} and \cite{MR4143728-merikoski}.

\Cref{thm:infinitude-aaabbb} immediately implies that $N(x;q,[a,\dots, a])$ tends to infinity whenever $a$ is a reduced residue class modulo $q$, as well as implying that for any two reduced residue classes $a,b$ modulo $q$, $N(x;q,[a,b])$ tends to infinity. Note that if $a$ and $b$ are reduced residue classes modulo $q$, then they are also admissible; however, \Cref{thm:infinitude-aaabbb} does not apply to all admissible residue classes $a$ and $b$, since for example $0$ mod $5$ or $9$ mod $27$ are admissible but not reduced. We believe that \Cref{thm:infinitude-aaabbb} should hold if $a$ and $b$ are only assumed to be admissible, and that our technique (along with an appropriate analog of McGrath's work in \cite{MR4498475-McGrath}) should suffice, but in the interest of brevity we address only this case.

Banks, Freiberg, and Turnage-Butterbaugh \cite{MR3316460-banks-freiberg-turnage-butterbaugh} derive the infinitude of $\pi(x;q,[a,\dots,a])$ from Maynard's work in \cite{MR3272929-Maynard-small-gaps} which shows that for any $m$, for large enough $k$, and for any admissible $k$-tuple $\{h_1, \dots, h_k\}$, there exist infinitely many $n$ such that $m$ of the $n+h_i$'s are simultaneously prime. The key idea is to choose a tuple $h_1, \dots, h_k$ where all elements are congruent mod $q$, and then choose $n$ to lie in a certain arithmetic progression. If every $h_i$ is congruent mod $q$, then any subset of $m$ of the $h_i$'s are also congruent mod $q$; for this reason, their method is restricted to constant patterns. Using instead McGrath's more refined result for sums of two squares, we can derive the infinitude of $N(x;q,\mathbf a)$ for patterns $\mathbf a$ that change value once, by having different congruence conditions for different bins. However, this technique at present does not extend to any specific pattern taking at least three values, since we cannot control how many sums of two squares appear in each bin. In other words, it is difficult to distinguish between a pattern of the form $[a,b,c]$ and one of the form $[a,b,b,c]$, since any bin may contain \emph{more} than one sum of squares.

However, the case of $r = 3$ can be addressed directly.
\begin{theorem}\label{thm:infinitude-abc}
    Let $q \ge 1$ be a modulus and let $a,b,$ and $c$ be admissible residue classes. Then $N(x;q,[a,b,c])$ tends to infinity with $x$.
\end{theorem}
This theorem relies crucially on Hooley's \cite{MR325557-Hooley-II} work on triple correlations of sums of two squares. Hooley showed that for any $h_1, h_2, h_3 \in \N$, 
\[\sum_{n \in \N} \mathbb{1}_{\mathbf E}(n+h_1)\mathbb{1}_{\mathbf E}(n+h_2)\mathbb{1}_{\mathbf E}(n+h_3)\]
is infinite, which he proved using the theory of quadratic forms. Hooley's result is close to what we are trying to prove, except that the sums of two squares $n+h_1, n+h_2,$ and $n+h_3$ are not necessarily consecutive, and do not necessarily lie in the correct congruence classes mod $q$. After restricting $n$ to an arithmetic progression, some of the results on quadratic forms used by Hooley no longer apply, and we need to rely instead on results on equidistribution of the roots of a quadratic polynomial. Precisely, we prove the following:
\begin{proposition}\label{prop-hk}
Let $q$ be a positive integer, and denote $q = \prod_{p}p^{\nu_p}$.
Let $a \in\ZZ/q\ZZ$ and let $h,k\in \NN$ be such that $a,a+h,a+k$ are all admissible mod $q$.
Assume that $\nu_p$ is even for all $p\equiv 3 \bmod 4$.
Assume also that $\nu_2$ is even and that
\begin{align}\label{eq:cond_mod2}
    \begin{split}
    &a\not \equiv 0 \bmod{2^{\nu_2 - 1}}\\
    &(a+h) \not \equiv 0 \bmod{2^{\nu_2 - 1}}\\
    &(a+k) \not \equiv 0 \bmod{2^{\nu_2 - 1}}.
    \end{split}
\end{align}
Then there are infinitely many integers $n$ such that $n\equiv a \bmod q$ and $n,n+h,n+k$ are all sums of two squares.
\end{proposition}

Hooley's work was also adapted in \cite{MR2764402-brudern-dietmann} to show that there are infinitely many triples of consecutive sums of two squares with specified gaps.
It would be interesting to further investigate the distribution of consecutive sums of two squares in arithmetic progressions. It is very likely possible to show that $N(x;q,\mathbf a) \gg N(x)$ whenever $\mathbf a$ satisfies the assumptions of \Cref{thm:infinitude-aaabbb}; the authors intend to address this in a follow-up paper. It is also not known whether $N(x;q,\mathbf a)$ tends to infinity for any pattern $\mathbf a$ of length $\ge 4$ not satisfying the assumptions of \Cref{thm:infinitude-aaabbb}.

Finally, the same setup could be applied to other sequences. The works of both Hooley and McGrath rely on using the representation function $r(n)$, defined as the number of ways to write $n = x^2 + y^2$, as a weighted version of $\mathbb{1}_{\mathbf E}$. It is not clear what, if anything, may be said about arbitrary sets arising by sieving the integers by a set of primes of density $\delta \le 1/2$, for example. 

\section{Results for two distinct congruence classes: proof of \autoref{thm:infinitude-aaabbb}}

We begin by recalling the result of McGrath \cite{MR4498475-McGrath} on sums of two squares in separate bins of an admissible tuple.

\begin{theorem}[McGrath \cite{MR4498475-McGrath}, Propositions 5.2 and 6.2 ]\label{thm:mcgrath}
Let $\mathcal H^* := \{h_1, h_2, \dots, h_k\}$ be an admissible tuple such that each $h_i$ is divisible by $4$. Fix real numbers $\theta_1,\theta_2$ subject to $0 < \theta_1+ \theta_2<1/18$ and define the constant
\[\Delta = \Delta(\theta_1,\theta_2) = \frac{\sqrt{2}(\pi +2)}{32\pi} \frac{1+\theta_1}{\sqrt{\theta_1\theta_2}}.\]

Fix a partition $\mathcal H^* = B_1 \cup B_2 \cup \cdots B_M$ where $|B_1| \ge 2\Delta^3$ and for $i \ge 2$, $|B_i| > 2^{7i}$. 

Then there exist elements $h_{a_1}, \dots, h_{a_M}$ and infinitely many integers $n$ such that $h_{a_j} \in B_j$ and $n+h_{a_j}$ is a sum of two squares for $1 \le j \le M$.
\end{theorem}

For what follows, we will need a slight variation of \Cref{thm:mcgrath}, where the $n$ guaranteed by the conclusion are required to be divisible by a fixed integer $g$.

\begin{theorem}[McGrath \cite{MR4498475-McGrath}, slight variation]\label{thm:mcgrath-but-linear-forms}
Let $g \ge 1$ be an odd integer, and let $\mathcal H^* := \{h_1, h_2, \dots, h_k\}$ be an admissible tuple such that $4|h_i$ and $(h_i,g) = 1$. Fix real numbers $\theta_1,\theta_2$ subject to $0 < \theta_1+ \theta_2<1/18$ and define the constant
\[\Delta = \Delta(\theta_1,\theta_2) = \frac{\sqrt{2}(\pi +2)}{32\pi} \frac{1+\theta_1}{\sqrt{\theta_1\theta_2}}.\]

Fix a partition $\mathcal H^* = B_1 \cup B_2 \cup \cdots B_M$ where $|B_1| \ge 2\Delta^3$ and for $i \ge 2$, $|B_i| > 2^{7i}$. 

Then there exist elements $h_{a_1},\dots, h_{a_M}$ and infinitely many integers $n$ such that $h_{a_j} \in B_j$ and $gn+h_{a_j}$ is a sum of two squares for $1 \le j \le M$. 
\end{theorem}
\begin{proof}
In the sieve set-up for the proof of \Cref{thm:mcgrath}, found on page 14 of \cite{MR4498475-McGrath}, all sieve sums are restricted to $n$ lying in a fixed residue class $v_0$ mod $W$ such that $(v_0 + h_i,W) = 1$ for each $i$, where $W$ is a product over small primes. The existence of $v_0$ is guaranteed by the admissibility of the set $\mathcal H$.

In our setting, by the admissibility of $\mathcal H^*$ and the fact that $(h_i,g) = 1$ for all $i$, we can further guarantee the existence of $v_0$ mod $[g,W]$, such that $(v_0 + h_i,W) = 1$ for each $i$ and $g|v_0$. The arguments of \cite{MR4498475-McGrath} yield elements $h_{a_1},\dots,h_{a_M}$ and infinitely many integers $n \equiv v_0 \mod [g,W]$ such that $h_{a_j} \in B_j$ and $n + h_{a_j}$ is a sum of two squares for all $j$; by our choice of $v_0$, each $n$ can be written as $gn'$, so we have infinitely many integers $n'$ such that $gn'+h_{a_j}$ is a sum of two squares for all $j$, as desired.
\end{proof}

We now strengthen this result further, by adding to the conclusion the constraint that the resulting sums of two squares are consecutive, following the ideas of \cite{MR3316460-banks-freiberg-turnage-butterbaugh} and \cite{MR3530450-Maynard-dense-clusters} in the prime case.

\begin{theorem}\label{thm:consecutive-mcgrath}
Let $g \ge 1$ be an odd integer, and let $\mathcal H := \{h_1, h_2, \dots, h_k\}$ be an admissible tuple of increasing positive integers such that $4|h_i$, and $(h_i,g) = 1$. Fix real numbers $\theta_1,\theta_2$ subject to $0 < \theta_1+ \theta_2<1/18$ and define the constant
\[\Delta = \Delta(\theta_1,\theta_2) = \frac{\sqrt{2}(\pi +2)}{32\pi} \frac{1+\theta_1}{\sqrt{\theta_1\theta_2}}.\]

Fix a partition $\mathcal H = B_1 \cup B_2 \cup \cdots B_M$ where $|B_1| \ge 2\Delta^3$ and for $i \ge 2$, $|B_i| > 2^{7i}$. 

Then there exist elements $h_{a_1},\dots, h_{a_M}$ and infinitely many integers $n$ such that $h_{a_j} \in B_j$ for all $j$, $gn+h_{a_j}$ is a sum of two squares for all $j$, and for any integer $t$ with $h_1 < t < h_k$, if $gn+t$ is a sum of two squares, then $t \in \mathcal H$. 
\end{theorem}
\begin{proof}
Let 
\begin{equation*}
S:=\{t \in \N: h_1 \le t \le h_k: t \not\in \{h_1,\dots,h_k\}\}.
\end{equation*}
Define a set of distinct primes $\{q_t:t\in S\}$ such that for each $t$, $q_t \equiv 3 \mod 4$ and $t \not\equiv h_j \mod q_t$ for all $t$. 
We also choose $q_t$ to satisfy $q_t> g$ so that $(g,q_t)$ = 1.
By the Chinese Remainder Theorem, there exists an integer $a$ with $ga+t \equiv q_t \mod q_t^2$ for all $t \in S$. In particular, $ga+t \equiv 0 \mod q_t$. Since $t\not\equiv h_j \mod q_t$, for all $t$ and for all $j$, $ga + h_j \not\equiv 0 \mod q_t$.

Define $\mathcal A(x):= \{gQx + ga+h_j\}_{j=1}^k$, where $Q :=\prod_{t\in S}q_t^2$. The tuple $\{ga + h_j\}_{j=1}^k$ satisfies the constraints of \Cref{thm:mcgrath} with respect to the modulus $gQ$. Moreover, by our choice of $Q$, $g(QN+a)+t \equiv q_t \mod q_t^2$, and thus $g(QN+a)+t$ is not a sum of two squares for any $h_1 < t < h_k$ with $t \not\in \mathcal H$. This completes the proof.
\end{proof}

We are now ready to prove \Cref{thm:infinitude-aaabbb}.

\begin{proof} [Proof of \Cref{thm:infinitude-aaabbb}]
For a tuple $\mathbf a$ of length $r$ satisfying the assumptions of \Cref{thm:infinitude-aaabbb}, let $r_a$ denote the number of $a$'s occurring in $\mathbf a$ and let $r_b$ denote the number of $b$'s occurring in $\mathbf a$.

Let $k$ be large enough that \Cref{thm:consecutive-mcgrath} applies with $M = r$. We will construct a tuple $\mathcal H = \{h_1, \dots, h_k\}$ as follows. First, we can split the index set $[1,k]$ into bins $B_1 \cup \cdots \cup B_M$ so that the size of the bins satisfies the hypotheses of \Cref{thm:consecutive-mcgrath}. Then for each $i$, choose $h_i > h_{i-1}$ such that $4|h_i$, such that $h_i \equiv a \mod q$ if $i \in B_j$ with $a_j = a$, such that $h_i \equiv b \mod q$ if $i \in B_j$ with $a_j = b$, and such that $\{qx+h_1,\dots, qx+h_k\}$ is an admissible tuple of linear forms. 

By \Cref{thm:consecutive-mcgrath}, there exist infinitely many tuples $(qn+h_{i_1},\dots, qn+h_{i_L})$ of consecutive sums of two squares such that at least one sum of two squares lies in each bin $B_j$; note that $L$ may be bigger than $M$, so that each bin may contain multiple sums of two squares. However, by the bin constraint, the first $r_a$ entries in each tuple $(qn+h_{i_1},\dots,qn+h_{i_L})$ must all be congruent to $a \mod q$ and the last $r_b$ must all be congruent to $b \mod q$, and every $qn+h_i$ is congruent to either $a$ or $b \mod q$. Moreover, by construction of $\mathcal H$, there is only one transition point from $h_{i_j} \equiv a \mod q$ to $h_{i_j}\equiv b\mod q$. Thus there is a sub-tuple of each $(qn+h_{i_1},\dots,qn+h_{i_L})$ satisfying the desired constraints, so there must be infinitely many.
\end{proof}

\section{3-Tuples: Proof of \autoref{thm:infinitude-abc}}
In this section we prove \Cref{thm:infinitude-abc}.
That is, we show that for any $q$, and any admissible $a,b,c \in \ZZ/q\ZZ$, the pattern $[a,b,c]$ appears in the sequence $(E_n \bmod q)$ infinitely often. 

\subsection{Notation}
Throughout this section we will use $(\cdot,\cdot)$ brackets to denote the greatest common divisor.
We will use $[\cdot,..., \cdot]$ brackets to denote a tuple of elements.
We will denote $n^k\| m$ to indicate that $n^k\mid m$ but $n^{k+1}\nmid m$.
For $a\in \ZZ/q\ZZ$, the reduction $a\bmod d$ is well defined for every $d\mid q$.
For any such $d$, we write $d\mid a$ if $a\equiv 0 \bmod d$.

\subsection{Proof of \autoref{thm:infinitude-abc}}
We first show how \Cref{thm:infinitude-abc} follows from \Cref{prop-hk}.
\begin{proof}
Let $q$ be a positive integer, and let $a,b,c$ be admissible mod $q$.
The congruence classes $a,b,c \bmod q$ can be lifted to admissible $a_2, b_2, c_2 \bmod q^2$.
We now show that these can then be lifted to admissible $a_3,b_3,c_3\bmod 4q^2$ which also satisfy the condition \eqref{eq:cond_mod2}.
Indeed, denote $q^2 = 2^\alpha r$ with $r$ odd and $a_2 \equiv 2^\beta \bmod 2^\alpha$ with $\beta \leq \alpha$.
If $\beta < \alpha$ then any admissible lift of $a_2$ to $a_3 \bmod 4q^2$ will satisfy \eqref{eq:cond_mod2}.
In the case $\beta = \alpha$, of the four possible lifts of $a_2$ to $a_3 \bmod 4q^2$ we consider the one which satisfies $a_3 \equiv 2^\alpha \bmod 2^{\alpha + 2}$.
This choice is both admissible and satisfies \eqref{eq:cond_mod2}.
The same argument applies to lifting $b_2,c_2$ to $b_3,c_3 \bmod 4q^2$.
We then fix $h<k\in\NN$ such that $a_3 + h\equiv b_3 \bmod 4q^2$ and $a_3 + k\equiv c_3 \bmod 4q^2$.

Let $\{p_i:1 \le i < k, i \ne h\}$ be a set of distinct primes with $p_i \nmid q$, $p_i>k$, and $p_i \equiv 3 \bmod 4$. Set
$$
T = 4q^2 \cdot \prod_{\substack{1\leq i < k \\ i\neq h}}p_i^2.
$$
Let $a_T\in \ZZ/T\ZZ$ be the congruence class satisfying $\alpha_T \equiv a_3 \bmod 4q^2$ and $a_T \equiv p_i -i  \bmod p_i^2$.
This choice of $a_T$ implies that $a_T, a_T + h, a_T + k$ are all admissible mod $T$, and that the requirements of \Cref{prop-hk} are satisfied.
Furthermore, for each $1\leq i \leq k, i\neq h,k$, $a_T + i$ is \emph{not} admissible mod $T$ because $a_T + i \equiv p_i \bmod p_i^2$. Thus if an integer $n$ satisfies $n\equiv a_T \bmod T$ and $n,n+h,n+k$ are sums of two squares, then they must be consecutive sums of two squares.

By \Cref{prop-hk}, there exist infinitely many $n \equiv a_T \bmod T$ such that $n, n+h$, and $n+k$ are all sums of two squares, and thus consecutive sums of two squares. Since $n\equiv a_T\bmod T$, we also have $n\equiv a \bmod q$, and thus by our choice of $h$ and $k$, $n+h\equiv b\bmod q$ and $n+k\equiv c\bmod q$.
It follows that there are infinitely many indices $t$ such that
\begin{equation*}
E_t \equiv a,\; E_{t+1} \equiv b,\; 
E_{t+2} \equiv c \; \bmod q. \qedhere
\end{equation*}
\end{proof}

We now prove \Cref{prop-hk}.
\begin{proof}
We use the notations from the statement of the proposition.

Since $a$ is admissible mod $q$, there exists a solution $x_0,y_0\in\ZZ$ to
\begin{equation}\label{eq-1}
x_0^2 + y_0^2 \equiv a \bmod q.
\end{equation}
Since $a + h$ is also admissible mod $q$, there exist $u,v\in \ZZ$ such that
\begin{equation}\label{eq-2}
(x_0 + u)^2 + (y_0 + v)^2 \equiv a +h\bmod q.
\end{equation}

We claim that $x_0,y_0,u,v$ can be chosen such that the following conditions are satisfied:
\begin{align}\label{eq-3}
\begin{split}
(u,v) &\mid 2(x_0,y_0),\\ 
(u,v) &\mid 2^{\nu_2 / 2 - 1}\prod_{\substack{p\mid q \\ p\equiv 3 \bmod 4}}p^{\nu_p / 2}. 
\end{split}
\end{align}

To show this, we consider \eqref{eq-1} and \eqref{eq-2} modulo different primes.
\begin{case}[Primes not dividing $q$]
We can always ensure that $(u,v)\mid q$.
Indeed, if $u_0,v_0$ is a solution, then for every $p\mid v_0$ such that $p\nmid q$, there exists $t_p$ with $p\nmid u_0+t_p q$.
By choosing $t\in\ZZ$ such that $t\equiv t_p\bmod p$ for all $p\mid v_0$, we get that $(u_0+tq,v_0)\mid q$.
Thus $v = v_0$, $u = u_0 + kq$ is a solution with $(u,v)\mid q$.

\end{case}

\begin{case}[Primes $p\equiv 1\bmod 4$ that divide $q$]
There exists a solution $x_0,y_0$ to \eqref{eq-1} such that $[x_0,y_0]\not\equiv [0,0]\bmod p$, even if $a \equiv 0 \bmod p$.
So we can assume that $p\nmid (x_0,y_0)$ for any $p \equiv 1\bmod 4$ with $p|q$.

If $p\nmid h$, then we must have $[u,v]\not\equiv [0,0]\bmod p$ which implies $p\nmid (u,v)$.
As for the case $p\mid h$, we can choose 
$$
[u,v]\equiv [-2x_0,-2y_0]\not\equiv [0,0] \bmod p.
$$
In either case, we can choose $u,v$ such that $p\nmid (u,v)$.
\end{case}

\begin{case}[Primes $p\equiv 3\bmod 4$ that divide $q$]
Assume that $p^\alpha\|(h,q)$ and that $p^{\beta}\|a$.
We know that $\beta$ is even since $a$ is admissible (and since $\nu_p$ is even).

Consider first the case where $\beta>\alpha$.
Then $p^\alpha\| a + h$. Since $a + h$ is also admissible, $\alpha$ is even as well.
Thus we can choose $x_0,y_0,u,v$ such that
\begin{align*}
p^{\beta/2} \| (x_0, y_0)\\
p^{\alpha/2}\| (u, v).
\end{align*}
Note that $\alpha/2 < \beta/2 \leq \nu_p/2$.

Next consider the case $\beta < \alpha$, so that $p^\beta \| a+h$ as well and $a \equiv a +h \bmod{p^{\beta + 1}}$.
We can choose $u,v$ such that 
$$
[u,v]\equiv[-2x_0,-2y_0]\bmod p^{\beta/2 + 1},
$$
which implies that
$$
a + h \equiv 
(x_0 + u)^2 + (y_0 + v)^2
\equiv x_0^2 + y_0^2 \equiv a \bmod p^{\beta + 1}.
$$
With this choice, we have $p^{\beta/2}\|(u,v)$ since $[-2x_0,-2y_0]\not\equiv [0,0]\bmod p^{\beta/2 + 1}$.

Finally consider the case $\alpha = \beta$.
In this case, we can choose $x_0,y_0,u,v$ such that
\begin{gather*}
p^{\alpha/2}\| (x_0,y_0)\\
p^{\alpha/2}\| (x_0 + u, y_0 + v).
\end{gather*}
This implies $p^{\alpha / 2} \mid (u,v)$.
We can further choose $u,v$ such that $p^{\alpha / 2} \| (u,v)$.
Otherwise, if $p^{\alpha / 2 + 1}\mid (u,v)$, then if we assume W.L.G that $p^{\alpha /2 + 1}\nmid y_0$ then we can replace $v$ by $-2y_0 - v$ which ensures $p^{\alpha/2 + 1}\nmid v$.

In each case, we found that if $p^\delta\|(u,v)$ then $p^\delta \mid (x_0,y_0)$ and $\delta \leq\nu_p/2$.
\end{case}

\begin{case}[The prime $2$]
Let $2^{\nu_2}\| q$, $2^{\alpha}\| a$, and $2^{\beta}\| a + h$.
Since we assumed that $a\not\equiv 0 \bmod{2^{\nu_2-1}}$, we know that $\alpha\leq \nu_2 -2$.
We can choose $x_0,y_0$ such that $2^{\lfloor \alpha/2\rfloor}\| (x_0,y_0)$ by choosing
$$
 [x_0,y_0]\equiv 
\begin{cases}
   [0,0]\bmod{2^{ \alpha/2}} & \alpha\text{ is even,} \\
    [2^{(\alpha - 1)/2},2^{(\alpha - 1)/2}]\bmod{2^{(\alpha - 1)/2 + 1}} & \alpha \text{ is odd.}
\end{cases}
$$

If $\beta < \alpha$, then we can choose $u,v$ such that 
$$
2^{\lfloor \beta/2\rfloor}\| (u,v),
$$
where $\lfloor\beta/2\rfloor < \lfloor\alpha/2\rfloor \leq \nu_2/2-1$.

Assume now that $\beta\geq \alpha$.
If $\alpha$ is even, then $\frac{x_0}{2^{\alpha/2}}$ and $\frac{y_0}{2^{\alpha/2}}$ must have opposite parity.
In this case, if $\beta =\alpha$ we can pick
$$
[u,v]\equiv[2^{\alpha/2 },2^{\alpha/2}]\bmod 2^{\alpha/2 + 1},
$$
and if $\beta>\alpha$ we can pick
$$
[u,v]\equiv[x_0,y_0]\bmod 2^{\alpha/2 + 1}.
$$

If $\alpha$ is odd then $\frac{x_0}{2^{(\alpha-1)/2}}$ and $\frac{y_0}{2^{(\alpha-1)/2}}$ must both be odd.
In this case, if $\beta>\alpha$ we can pick
$$
[u,v]\equiv[x_0,y_0]\bmod{2^{(\alpha-1)/2 + 1}},
$$
and if $\alpha=\beta$ we can pick
$$
[u,v]\equiv[2^{(\alpha-1)/2 +1},2^{(\alpha-1)/2 +1}]\bmod{2^{(\alpha-1)/2 + 2}}.
$$
In this case, we have $2^{(\alpha-1)/2 +1}\|(u,v)$ and $2^{(\alpha-1)/2}\|(x_0,y_0)$; this  is why the extra factor of $2$ is needed in \eqref{eq-3}.
The condition $\alpha\leq \nu_2-2$ ensures that we still have 
$$
1 + \left((\alpha-1)/2 +1\right) \leq \nu_2
$$
which is required since we want $2(u,v)\mid q$.
\end{case}

\setcounter{case}{0}

To summarize, we have now found $x_0,y_0,u,v\in\ZZ$ which simultaneously satisfy \eqref{eq-1},\eqref{eq-2}, and \eqref{eq-3}.
Write $T := \frac{q}{2(u,v)},$ so that by \eqref{eq-3}, $q\mid T^2$ and $q\mid 2(x_0,y_0)T$.
For $r,s \in \ZZ$, define
\begin{align*}
x &= x(r) = x_0 + Tr, \\
y &= y(s) = y_0 + Ts,
\end{align*}
so that for any $n$ with $n = x^2 + y^2$,
$$
n = x_0^2 + y_0^2 + 2T(rx_0 + sy_0) + T^2(r^2 + s^2) 
\equiv x_0^2 + y_0^2 \equiv a \bmod q.
$$
For suitable choices of $r,s$ we will have $n+h = (x+u)^2 + (y+v)^2$.
For this to hold, we need the following equality to hold:
\begin{align}\label{eq-5}
\begin{split}
h &= (x+u)^2 + (y+v)^2 - x^2 - y^2 \\
&= u^2 + v^2 + 2u\left(x_0 + r T\right) + 2b\left(y_0 + s T\right) \\
& = u^2 + v^2 + 2(u x_0 + v y_0) + q\left(\frac{u}{(u,v)}r + \frac{v}{(u,v)}s\right).
\end{split}
\end{align}
This equation holds mod $q$, so there is some choice of $r_0,s_0$ which solves \eqref{eq-5} in $\ZZ$. We therefore write:
\begin{align*}
\begin{split}
r = r(t) = r_0 + \frac{v}{(u,v)}t,\\
s = s(t) = s_0 - \frac{u}{(u,v)}t.
\end{split}
\end{align*}

Then $n = x^2 + y^2 = (x_0 + T r(t))^2 + (y_0 + Ts(t))^2$, which implies that
$$
n = x^2 + y^2 \equiv a\bmod q
$$ 
and 
$$
n+h = (x+u)^2 + (y+ v)^2.
$$

Our goal is to show that there are infinitely many $t$'s such that $n+k = n(t) + k$ is a sum of two squares.
Every such $t$ determines an integer $n = n(t)$ such that $n,n+h,n+k$ are all sums of two squares, and $n\equiv a\bmod q$.
Furthermore, for every $n_0$ there are only finitely many $t$'s that $n(t) = n_0$, since the number of representations of $n_0$ as a sum of two squares is finite (and different $t's$ give different representations of $n_0$ as a sum of two squares). Thus if there are infinitely many such $t$'s, there are infinitely many $n \equiv a \bmod q$ with $n, n+h,$ and $n+k$ all sums of two squares.

Expanding $n(t)$ gives
\begin{align*}
n(t) + k 
=& k + 
\left(x_0 + \frac{q}{2(u,v)}\left(r_0 +\frac{v}{(u,v)}t\right)\right)^2
+ \left( y_0 + \frac{q}{2(u,v)}\left(s_0 -\frac{u}{(u,v)}t\right)\right)^2 \\
 =& k + \left(x_0 + \frac{q}{2(u,v)} r_0\right)^2 
 + 
\left(y_0 + \frac{q}{2(u,v)}s_0\right)^2 \\
&+ \frac{q}{(u,v)}\left(
\left(x_0 + \frac{q}{2(u,v)} r_0\right)\frac{v}{(u,v)} 
-
\left(y_0 + \frac{q}{2(u,v)}s_0\right)\frac{u}{(u,v)} 
\right)t \\
& +\left(\frac{q}{2(u,v)}\right)^2 \frac{u^2 + v^2}{(u,v)^2} t^2.
\end{align*}

We call this polynomial $F(t) = At^2 + Bt + (C+k)$.
We want to show that 
\begin{equation}\label{eq-6}
\sum_{t\leq X}\mathbb{1}_{\mathbf E}(F(t)) \xrightarrow{X\rightarrow \infty}\infty.
\end{equation}

We first show that there are no local obstructions.
The leading coefficient of $F(t)$ is 
$$
A = \left(\frac{q}{2(u,v)}\right)^2 \frac{u^2 + v^2}{(u,v)^2}.
$$
For every prime power $p^\alpha \nmid q$ with $p\equiv 3\bmod 4$, $A \not\equiv 0 \mod q$, so $F(t)$ is not the zero polynomial mod $p^\alpha$. 

For every prime power $p^\alpha\mid q$ with $p\equiv 3\bmod 4$, since $n(t) \equiv a \mod q$, $n(t) \equiv a \mod p^\a$, and thus the equation $F(t) = z^2 + w^2 \bmod{p^\alpha}$ becomes
$$
z^2 + w^2 = k + a \bmod{p^\alpha}
$$
which is solvable since $k+a$ is admissible mod $q$.
Since $q$ is divisible by $p$ an even number of times, there are no local obstructions at any prime $p\equiv 3\bmod 4$.

We now consider local obstructions at powers of $2$.
Such obstructions have the form $F(t) \equiv 3\cdot 2^{\alpha - 2}\bmod {2^\alpha}$ for all $t\bmod {2^\alpha}$ for some $\alpha$.
However, we see that $F(0) \equiv a+ k \bmod 2^{\nu_2}$.
The fact that $a+ k$ is admissible mod $2^{\nu_2}$, and that $a+k\not\equiv 0 \bmod{2^{\nu_2 -1}}$ implies that there are no obstructions at any power of $2$.

We now consider the discriminant of $F$, defined by
$$
\text{disc}(F) = B^2 - 4A(C+k) = B^2 - 4AC - 4Ak.
$$
Since $n(t) = F(t) - k$ is given as a sum of two squares, we get that $B^2 - 4AC$ has the form $-\eta^2$ for some integer $\eta$, and is thus non-positive.
Furthermore, $A$ is a sum of two squares by its definition, so $A\geq 0$, and $k>0$ since $k\in \NN$.
It follows that $B^2 - 4A(C+k) \le 0$.

If $\text{disc}(F)$ has the form $-\eta^2$ for some $\eta\in\ZZ$, then by completing the square, it is easy to see that $F(t)$ is a sum of two squares for all $t$.
Otherwise, if $\text{disc}(F)\neq -\eta^2$, we get from \cite[Theorem 2]{MR476673-friedlander-iwaniec} that 
$$
\sum_{t\leq X}\mathbb{1}_{\mathbf E}(F(t)) \gg \frac{X}{\sqrt{\log X}}.
$$
In either case, there are infinitely many $t$'s such that $n(t)+k$ is a sum of two squares.

\end{proof}

\bibliographystyle{amsplain}
\bibliography{bibliography}

\end{document}